\documentclass[11pt]{amsart}

\pdfoutput=1

\usepackage{amsmath}
\usepackage{fullpage}
\usepackage{xspace}
\usepackage[psamsfonts]{amssymb}
\usepackage[latin1]{inputenc}
\usepackage{graphicx,color}
\usepackage{hyperref}
\usepackage{graphicx}

\usepackage{amsmath}%
\usepackage{amsthm}%
\usepackage{amscd}
\usepackage{amsfonts}%
\usepackage{amssymb}%
\usepackage{graphicx}

\usepackage{mathrsfs}

\usepackage{tikz}
\usetikzlibrary{matrix,arrows}

\usepackage{tikz-cd}

\newtheorem{theorem}{Theorem}[section]

\newtheorem{lemma}[theorem]{Lemma}

\theoremstyle{remark}

\numberwithin{equation}{section}

\newcommand{\Hcal}{\mathscr{H}}

\newcommand{\Lcal}{\mathscr{L}}

\newcommand{\Rcal}{\mathscr{R}}

\newcommand{\Ycal}{\mathscr{Y}}

\newcommand{\Z}{\mathbb{Z}}
\newcommand{\C}{\mathbb{C}}

\newcommand{\Q}{\mathbb{Q}}

\newcommand{\Gal}{\mathrm{Gal}}

\makeatletter
\@namedef{subjclassname@2020}{%
  \textup{2020} Mathematics Subject Classification}
\makeatother

  \DeclareFontFamily{U}{wncy}{}
    \DeclareFontShape{U}{wncy}{m}{n}{<->wncyr10}{}
    \DeclareSymbolFont{mcy}{U}{wncy}{m}{n}
    \DeclareMathSymbol{\Sha}{\mathord}{mcy}{"58}

\begin{document}
\title[]{Hilbert's tenth problem for rings of holomorphic functions of bounded order}

\author{Hector Pasten}
\address{ Departamento de Matem\'aticas,
Pontificia Universidad Cat\'olica de Chile.
Facultad de Matem\'aticas,
4860 Av.\ Vicu\~na Mackenna,
Macul, RM, Chile}
\email[H. Pasten]{hector.pasten@uc.cl}%

\thanks{Supported by ANID Fondecyt Regular grant 1230507 from Chile.}
\date{\today}
\subjclass[2020]{Primary: 03B25 ; Secondary: 30D15, 11U05} %
\keywords{Hilbert's tenth problem, holomorphic functions, finite order, Pell equations}%

\begin{abstract} One of the main open problems in the context of extensions of Hilbert's tenth problem (HTP) is the case of the ring of complex entire functions in one variable. In the direction of an answer, for every $\rho\ge 0$, we give a negative solution to the analogue of HTP in the ring of complex entire functions in one variable of growth order at most $\rho$. 
\end{abstract}

\maketitle



\section{Introduction} 

Hilbert's tenth problem asked for an algorithm that decides the solvability in $\Z$ of polynomial Diophantine equations with coefficients in $\Z$. This problem was negatively solved in 1970 by work of Davis, Putnam, Robinson, and Matiyasevich \cite{DPR, Matiyasevich}, and after this breakthrough there has been considerable research around extensions of the problem to other rings. One of the main open cases is that of the ring $\Hcal$ of complex holomorphic functions in one variable over $\C$. While the literature on Hilbert's tenth problem for rings of \emph{algebraic} functions is abundant, the results for subrings of $\Hcal$ are rather limited due to the presence of \emph{transcendental} functions. These results in the transcendental setting basically concern the cases of  exponential polynomials \cite{CGFPPV} and rings of holomorphic functions of finite order with lacunary Taylor expansion  \cite{GFPlacunary}. 

 On the other hand, the analogue of Hilbert's tenth problem for non-Archimedian rigid analytic functions in one variable (the non-archimedian analogue of $\Hcal$) has been solved negatively \cite{LipPhe, GFPpos}, while the version for rings of complex holomorphic functions on at least two variables has also been settled \cite{PheidasVidaux}.
 
 In this work we give a negative solution to the analogue of Hilbert's tenth problem for rings of complex holomorphic functions of bounded order. Before stating our main result, let us recall the notion of order of an entire function. Given a holomorphic function $f\in \Hcal$, the order $\alpha(f)$ of $f$ is the infimum of all real numbers $\alpha\ge 0$ such that as $r\to \infty$ we have the bound
 $$
 \log\max_{|z|\le r} |f(z)|\le O(r^\alpha)
 $$
(if no $\alpha$ satisfies this bound, one lets $\alpha(f)=+\infty$.) For any given $\rho\ge 0$, we define 
$$
\Hcal^\rho=\{f\in \Hcal: \alpha(f)\le \rho\}
$$ 
and we observe that $\Hcal^\rho$ is a ring that contains $\C[z]$. In particular, $\Hcal^\rho$ is an $\Lcal_z$-structure for the language $\Lcal_z=\{0,1,z,+,\cdot, =\}$. With this notation, our main results are:

\begin{theorem}\label{ThmMain1} Let $\rho \ge 0$. The set $\Z\subseteq \Hcal^\rho$ is positive existentially $\Lcal_z$-definable in $\Hcal^\rho$. 
\end{theorem}

\begin{theorem}\label{ThmMain2} Let $\rho \ge 0$. The positive existential theory of $\Hcal^\rho$ over the language $\Lcal_z$ is undecidable. Thus, the analogue of Hilbert's tenth problem for $\Hcal^\rho$ with coefficients in $\Z[z]$ has a negative solution.
\end{theorem}

A standard reduction argument shows that Theorem \ref{ThmMain2} follows from Theorem \ref{ThmMain1} due to the negative solution to Hilbert's tenth problem over $\Z$.  Thus, we will focus on proving Theorem \ref{ThmMain1}.

We remark that the rings $\Hcal^\rho$ naturally occur in complex analysis. For instance, it is a theorem of Brody \cite{Brody} that if $X$ is a complex compact manifold, one has that there is a non-constant complex holomorphic map $f:\C\to X$ if and only if there is such a map of order at most $2$.



\section{Chebyshev polynomials and a Pell equation} 

Consider the following Pell equation over $\C[z]$:
\begin{equation}\label{EqnPell}
X^2-(z^2-1)Y^2=1.
\end{equation}
Let $w=\sqrt{z^2-1}$, which is algebraic over $\C[z]$. For each integer $n$ we define the polynomials $x_n,y_n\in \C[z]$ by the identity
\begin{equation}\label{EqnDef1}
x_n+y_nw = (z+w)^n.
\end{equation}
Multiplying the previous expression by its Galois conjugate (replacing $w$ by $-w$) one obtains that $(x_n,y_n)$ is a solution for \eqref{EqnPell} for each $n\in \Z$. The following result, due to Denef, is by now standard (see \cite{DenefPoly,DenefPos}, and see Lemma 2.3 in \cite{PheidasZahidi} for an exposition):
\begin{theorem}[Denef]\label{ThmDenef}
The only solutions of \eqref{EqnPell} in $\C[z]$ are given by $(\pm x_n,y_n)$ for $n\in \Z$. Furthermore:
\begin{itemize}
\item[(i)] $x_n,y_n\in \Z[z]$ for each $n$;
\item[(ii)] $\deg(x_n)=|n|$ for each $n$;
\item[(iii)] $y_n(1)=n$ for each $n$; and
\item[(iv)] $k$ divides $n$ in $\Z$ if and only if $y_k$ divides $y_n$ in $\C[z]$.
\end{itemize}
\end{theorem}

One can check (see \cite{Demeyer}) that $T_n=x_n$ and $U_{n-1}=y_n$ are the classical Chebyshev polynomials of the first and second kind respectively. Thus, and alternative definition for $x_n$ and $y_n$ is given by the identities
\begin{equation}\label{EqnChebyshev}
x_n(\cos\theta)=\cos(n\theta) \quad \mbox{and}\quad y_n(\cos\theta)\sin \theta=\sin(n\theta).
\end{equation}


\section{Holomorphic solutions of the Pell equation} 

The solutions of \eqref{EqnPell} in the ring $\Hcal$ of entire functions were completely classified in \cite{CGFPPV}. The result is the following one (again, $w=\sqrt{z^2-1}$):

\begin{theorem}[Classification of holomorphic solutions of \eqref{EqnPell}] \label{LemmaClassification}
The solutions $(x,y)$ of \eqref{EqnPell} in $\Hcal$ are precisely given by entire functions $x$ and $y$ defined by an expression of the form
$$
x+yw = \epsilon\cdot (z+w)^n\exp(hw)
$$
for an integer $n$, a function $h\in \Hcal$, and some $\epsilon\in \{-1,1\}$.
\end{theorem}

Let $\lfloor t\rfloor$ be the integral part of a real number $t$, and for convenience let us define the degree of the zero polynomial as $-1$. We need the following growth estimate.

\begin{theorem}[Growth of holomorphic solutions of \eqref{EqnPell}]\label{LemmaGrowth} Let $(x,y)$ be a holomorphic solution of \eqref{EqnPell} defined by
$$
x+yw = \epsilon\cdot (z+w)^n\exp(hw)
$$
for some $\epsilon\in\{-1,1\}$, $n\in \Z$, and $h\in \Hcal$. Let $\rho\ge 0$. If the order of $y$ satisfies $\alpha(y)\le \rho$ then $h$ is a polynomial of degree at most $\lfloor\rho\rfloor-1$.
\end{theorem}
The proof of this result is similar to the proof of Theorem 3.4 in \cite{CGFPPV}; the straightforward modification of the argument is left to the reader.

A general fact about Pell equations is that their solutions form an abelian group. In the case of \eqref{EqnPell}, let us define the operation $\oplus$ in its set of holomorphic solutions by
$$
(u,v)\oplus (x,y)=(f,g)\quad \mbox{ where }\quad f+gw=(u+vw)(x+yw).
$$
Explicitly,
\begin{equation}
\begin{cases}
f&= ux+(z^2-1)vy,\\
g&= uy + vx.
\end{cases}
\end{equation}
At this point let us introduce the notation $(x,y)^{\oplus k}$ for an integer $k$, to denote the multiplication by $k$ in the group structure of the solutions of \eqref{EqnPell}.

Then one checks the following:
\begin{lemma}[Group structure]\label{LemmaGroupPell} Let $\Rcal$ be a sub-ring of $\Hcal$ containing $\Z[z]$. The solutions in $\Rcal$ of \eqref{EqnPell} form an abelian group under the operation $\oplus$. Furthermore, for all integers $m,n$ we have
$$
(x_m,y_m)\oplus(x_n,y_n)=(x_{m+n},y_{m+n}).
$$
In particular $(x_k,y_k)=(z,1)^{\oplus k}$.
\end{lemma}


\section{Rational values of certain trigonometric polynomials} 

Let us recall some basic facts about the maximal totally real subfield of a cyclotomic field. For a positive integer $k$ let $L_k=\Q(\cos(2\pi/k))$. Then $L_k$ is the maximal totally real subfield of the cyclotomic field $\Q(\exp(2\pi i/k))$ and for each divisor $d|k$ we have the inclusion $L_d\subseteq L_k$. The extension $L_k/\Q$ is Galois and  $\Gal(L_k/\Q)\simeq G_k:=(\Z/k\Z)^\times/(\pm 1)$. Thus, for $k\ge 3$ we have $[L_k:\Q]=\varphi(k)/2$ where $\varphi$ is Euler's function. The isomorphism $G_k\to \Gal(L_k/\Q)$, $a\mapsto \sigma_a$ is explicitly given by the action  on the generator $\cos(2\pi/k)$, which is $\sigma_a(\cos(2\pi/k))=\cos(2\pi a/k)$. 

With these preliminaries, let us discuss the main result of this section.
\begin{lemma}\label{LemmaCos} Let $m\ge 1$ be an integer and let $k=pq$ where $p$ and $q$ are primes satisfying $p>m+1$ and $q>2(m+1)$. Let $F(z)\in\C[z]$ be a non-constant polynomial and consider the function $f(t)=F(\cos(2\pi t))$. Suppose that $f(a/k)\in \Q$ for each $a=1,2,...,k-1$. Then $\deg(F)>m$.
\end{lemma}
\begin{proof}
For the sake of contradiction, let us suppose that $\deg(F)\le m$.

The matrix 
$$
\left[\cos^b(\pi a/q)\right]_{\substack{1\le a\le (q-1)/2\\0\le b\le m}}
$$
has maximal rank (equal to $m+1$) because $q\ge 2m+3$ and it contains a Vandernonde sub-matrix.  
From this fact and the condition $f(a/q)=f(ap/k)\in \Q$ for each $a=1,...,q-1$, it follows that $F(z)\in L_q[z]$.

Let $\gamma:=F(\cos(2\pi/k))=f(1/k)\in \Q$. Let us consider the reduction map $G_{k}\to G_q$ and let $K$ be its kernel. Then $\#K=(\phi(k)/2)/(\phi(q)/2)=p-1$. Note that the Galois action of $K$ on $L_k$ fixes the elements of $L_q$ and, in particular, it fixes the coefficients of the polynomial $F(z)$. Therefore, letting  $K$ act on both sides of the equation $\gamma=F(\cos(2\pi/k))$ we deduce that there is a set $S\subseteq \{1,...,(k-1)/2\}$ with $\#S=p-1$ such that $F(\cos(2\pi a/k))=\gamma$ for each $a\in S$. This produces $p-1$ different zeros of the polynomial $F(z)-\gamma$, which is not the zero polynomial because $F(z)$ is non-constant. However, $p-1>m\ge \deg(F)$; contradiction.
\end{proof}


\section{Definability of the integers} 

\begin{lemma}\label{LemmaKey} Let $\rho\ge 0$ be a real number, let $m=2(\lfloor\rho\rfloor+1)$, let $p,q$ be primes with $p>m+1$ and $q>2(m+1)$, and let $k=pq$. Let $(x,y)$ be a solution of the Pell equation \eqref{EqnPell} in $\Hcal^\rho$ such that the following two conditions hold:
\begin{itemize}
\item[(i)] there is a solution $(s,t)$ of \eqref{EqnPell} in $\Hcal^\rho$ with $(s,t)^{\oplus k} = (x,y)$; and
\item[(ii)] the polynomial $y_k$ divides $y$ in $\Hcal^\rho$. 
\end{itemize}
Then $y=y_{kn}$ for some integer $n$. 

Furthermore, for every integer $\ell$ we have that the polynomial $y=y_{k\ell}$ belongs to $\Hcal^\rho$ and it satisfies (i) and (ii).
\end{lemma}
\begin{proof} By Lemma \ref{LemmaClassification} and the hypothesis (i) we have that $x,y$ satisfy
$$
x+yw=(s+tw)^k=\epsilon\cdot (z+w)^{k\ell}\exp(hw)=\epsilon\cdot (x_{k\ell}+y_{k\ell}w)\exp(hw)
$$
for some  $\epsilon\in\{-1,1\}$, $h\in \Hcal$, and $\ell, k\in \Z$. Furthermore, by Lemma \ref{LemmaGrowth} and the fact that $y\in \Hcal^\rho$, we have that $h$ is a polynomial of degree $\deg h\le \lfloor \rho\rfloor-1$. It suffices to show that $h=0$ because $\epsilon\cdot y_{k\ell} = y_{\epsilon k\ell}$. For the sake of contradiction, let us assume that $h$ is not the zero polynomial.

Let $g(z)=(2\pi i)^{-1}h(z)$ where $i=\sqrt{-1}$. Then 
$$
\exp(hw) = \cos(2\pi gw) + i\cdot  \frac{\sin(2\pi gw)}{w}\cdot w
$$
and by considering Taylor expansions we observe that the functions
$$
u=\cos(2\pi gw) \quad\mbox{and}\quad v=\frac{\sin(2\pi gw)}{w}
$$
belong to $\Hcal$. Hence we have
$$
x+yw=\epsilon\cdot (x_{k\ell}+y_{k\ell}w)(u+ivw)
$$
from where it follows that
$$
y = \epsilon\cdot (uy_{k\ell} +ivx_{k\ell}).
$$
The hypothesis (ii) and Theorem \ref{ThmDenef} give that $y_k$ divides $vx_{k\ell}$. Since $y_k|y_{k\ell}$ and for all $n$ the polynomials $x_n,y_n$ have no common zeros (they satisfy Equation \ref{EqnPell}) we deduce that in fact $y_k$ divides $v$. 

It follows that for every $z_0\in \C$ with $y_{k}(z_0)=0$ we have that $v(z_0)=0$, hence, $2g(z_0)\sqrt{z_0^2-1}$ is an integer (note that integrality is well-defined regardless of the sign choice of $\sqrt{z_0^2-1}$.) Let 
$$
F(z)=4g(z)^2(z^2-1)\in \C[z],
$$ 
then for such a $z_0$ we have that $F(z_0)\in \Z$. Furthermore, since $h$ is not the zero polynomial, we have that $F$ is non-constant.

By \eqref{EqnChebyshev} we see that for all $a\in \{1,2,...,k-1\}$ one has
$$
y_k(\cos(2\pi a/k)) = \frac{\sin(2\pi a)}{\sin(2\pi a/k)} =0.
$$
Indeed, $k=pq$ is odd, so $2\pi a/k\notin \Z\cdot \pi$ for any $a=1,2,...,k-1$. In this way we deduce that
$$
F(\cos(2\pi a/k))\in \Z
$$
for each $a=1,2,...,k-1$. Lemma \ref{LemmaCos} now shows that $\deg F> m = 2(\lfloor\rho\rfloor+1)$ which gives $\deg h =\deg g> \lfloor\rho\rfloor$. However, we had $\deg h\le \lfloor\rho\rfloor-1$, which is a contradiction. This proves $h=0$, hence, $y = y_{\epsilon k\ell}$ as we wanted.

Finally, the last part of the theorem follows from Theorem \ref{ThmDenef} and Lemma \ref{LemmaGroupPell}.
\end{proof}

We are now ready to prove Theorem \ref{ThmMain1}.


\begin{proof}[Proof of Theorem \ref{ThmMain1}] Let $\rho\ge 0$ be a real number. Let $m=2(\lfloor \rho \rfloor +1)$, let $p,q$ be primes with $p>m+1$ and $q>2(m+1)$, and let $k=pq$. This choice of $k$ depends on $\rho$, but from now on it will be fixed. So, our $\Lcal_z$-formulas giving the necessary definitions may depend on $k$.

By Lemma \ref{LemmaKey}, the following set of polynomials is positive existentially $\Lcal_z$-definable in $\Hcal^\rho$:
$$
\Ycal_k = \{y_{kn} : n\in \Z\}.
$$
This is because conditions (i) and (ii) in Lemma \ref{LemmaKey} can be expressed by positive existential $\Lcal_z$-formulas. 

By Picard's theorem there is no non-constant holomorphic map $\C\to C$ to a complex algebraic curve $C$ of geometric genus at least $2$. Thus, the set $\C\subseteq \Hcal^\rho$ is positive existentially $\Lcal_z$-defined by the formula
$$
c(t):\quad \exists s, t^2=s^5+1
$$
because the hyperelliptic curve $y^2=x^5+1$ has genus $2$. 

It follows that the following expression
$$
\Phi(t):\quad \exists y\in \Ycal_k,\exists f, ((z-1)f=y-t)\wedge c(t)
$$
can be expressed by a positive existential $\Lcal_z$-formula and in $\Hcal^\rho$ it defines the set
$$
\{t\in \C : t=y_{kn}(1) \mbox{ for some }n\in \Z\} \subseteq \Hcal^\rho.
$$
By Theorem \ref{ThmDenef} we see that the previous set is $k\cdot \Z$, and the required definition of $\Z$ follows.
\end{proof}


\section{Acknowledgments}

Supported by ANID Fondecyt Regular grant 1230507 from Chile. 


\end{document}